\documentclass[11pt,leqno]{amsart}
\usepackage[utf8]{inputenc}
\usepackage{amsfonts,amssymb,hyperref}
\usepackage{mathtools}
\usepackage{enumerate,float}
\usepackage{cleveref}
\usepackage{mathrsfs}
\usepackage{soul}

\usepackage[normalem]{ulem}

\numberwithin{equation}{section}

\theoremstyle{plain}
\newtheorem{lemma}{Lemma}[section]
\newtheorem{proposition}[lemma]{Proposition}
\newtheorem{theorem}[lemma]{Theorem}
\newtheorem{corollary}[lemma]{Corollary}

\theoremstyle{remark}
\newtheorem{question}[lemma]{Question}
\def\N{\mathbb{N}}
\def\Z{\mathbb{Z}}

\def\C{\mathcal{C}}

\def\x{\textbf{x}}
\def\g{\textbf{g}}
\def\h{\textbf{h}}
\def\n{\textbf{n}}

\DeclareMathOperator{\cl}{cl}
\DeclareMathOperator{\pcl}{cl_p}
\DeclareMathOperator{\vcl}{cl_{v}^{\hspace{-1pt}\textit{H}}}

\DeclareMathOperator{\vvcl}{cl_{v}}

\DeclareMathOperator{\Gvcl}{cl_{v}^{\hspace{-1pt}\textit{G}}}

\DeclareMathAlphabet{\mathpzc}{OT1}{pzc}{m}{it}
\newcommand{\rcal}{\mathpzc{r}}

\title[Strong conciseness and equationally Noetherian groups]{Strong conciseness and equationally Noetherian groups}

\author[I. de las Heras]{Iker de las Heras} 
\address{Iker de las Heras: Department of Mathematics, Euskal Herriko Unibertsitatea UPV/EHU, Basque Country}
\email{iker.delasheras@ehu.eus}
\date{}

\author[A. Zozaya]{Andoni Zozaya} 
\address{Andoni Zozaya: Fakulteta za Matematiko in Fiziko, University of Ljubljana, Slovenia}
\email{andoni.zozaya@fmf.uni-lj.si}
\date{}

\makeatletter
\@namedef{subjclassname@2020}{\textup{2020} Mathematics Subject
  Classification}
\makeatother
\thanks{The first author is supported by the Spanish Government, grant PID2020-117281GB-I00 (partly with FEDER funds), and by the Basque Government, grant IT483-22.
During a large part of this reasearch, he was also supported by the European Union via the Marie Skłodowska-Curie Actions (MSCA), grant HE-MSCA-PF-GF22/02 101067088.
The second author is supported by the Slovenian Research Agency, programs P1-0222, P1-0294 and grants J1-50001, J1-4351, N1-0216.}

\begin{document}

\begin{abstract}
A word $w$ is said to be concise in a class of groups if, for every $G$ in that class such that the set of $w$-values $w\{G\}$ is finite, the verbal subgroup $w(G)$ is also finite.
In the context of profinite groups, the notion of strong conciseness imposes a more demanding condition on $w$, requiring that $w(G)$ is finite whenever $|w\{G\}|< 2^{\aleph_0}$. 
We investigate the relation between these two properties and the notion of equationally Noetherian groups,
by proving that in a profinite group $G$ with a dense equationally Noetherian subgroup, $w\{G\}$ is finite whenever $|w\{G\}|< 2^{\aleph_0}$.
Consequently, we conclude that every word is strongly concise in the classes of profinite linear groups, pro-$\C$ completions of residually $\C$ linear groups and pro-$\C$ completions of virtually abelian-by-polycyclic groups, thereby extending well-known conciseness properties of these classes of groups.
\end{abstract}
\maketitle


\section{Introduction} 

A \emph{word} $w$ in $k$ variables is an element of the free group on $k$ generators.
For every group $G$, we can interpret a word as a map $ w \colon G^k \rightarrow G$ defined by substitution of group elements for variables.
Verbal problems in groups study the relationship between the set of \emph{$w$-values} in $ G $, defined as
\[
w\{G \}  = \{ w(g_1, \dots, g_k) \mid g_i \in G\},
\]
and the \emph{verbal subgroup} of $ w $ in $ G $, defined as $w(G) = \langle w \{ G \} \rangle $. Consistent with the terminology introduced by P. Hall, we say that $ w $ is \emph{concise} in a class of groups $\mathfrak{C}$ if, for all $G$ in $\mathfrak{C}$, the verbal subgroup $w(G)$ is finite whenever $ w\{G\} $ is finite.

In the context of profinite groups, this notion has been recently extended (\textit{cf.}~\cite{DKS, DMS}): we say that $ w $ is \emph{strongly concise} in a class $ \mathfrak{C} $ of profinite groups if, for any $ G $ in $ \mathfrak{C}$, the subgroup $ w(G) $ is finite whenever $ |w\{G\}| < 2^{\aleph_0} $. The study of conciseness and strong conciseness has  drawn significant attention in recent years (\textit{cf.}~\cite{Be, DPS, FePi, KS, Pinto, A2}).

Even though P.~Hall conjectured that every word is concise in the class of all groups, this was proved false by Ivanov in~\cite{Ivanov}, by building an example of a word that is not concise in a certain non-residually finite group.
However, it remains unknown whether all words are concise -- or even strongly concise -- in the class consisting of all profinite groups.

Several classes of groups are known to have the property that every word is concise on them.
For instance, every word is concise in the class of linear groups (Merzjalkov~\cite{Me672}), or in the class of residually finite groups all of their quotients are residually finite (Turner-Smith~\cite{TS}).
In particular, it follows from the latter that every word is concise in the classes of virtually abelian-by-nilpotent or virtually abelian-by-polycyclic groups.

 Among the aforementioned classes, strong conciseness 
 has been so far only proved for virtually nilpotent profinite groups, meaning that every word is strongly concise in this class of groups (\textit{cf.}~\cite[Theorem 1.1]{De23} and \cite[Theorem 1.2]{DKS}).
 In this article we establish the same property for profinite analogues of two previously mentioned classes: profinite linear groups (profinite groups with an embedding into matrix groups), and pro-$\C$ completions of either linear or abelian-by-polycyclic groups. Here and henceforth, $\C$ refers to a \emph{formation} of finite groups -- a class of finite groups closed under taking quotients and subdirect products of its members. This study is based on the relationship between conciseness, strong conciseness, and the notion of equationally Noetherian groups, a key concept in algebraic geometry over groups; we refer to Section~\ref{sec: preliminaries} for the definition. We remark that the connection between equational Noetherianity and residual finiteness has already been studied in~\cite{Val}.

 \medskip

The main result of this paper is the following:

\begin{theorem}
\label{thm: eq noetherian}
Let $G$ be a profinite group and let $w$ be a word such that $|w\{G\}|<2^{\aleph_0}$.
Suppose that $G$ has an equationally Noetherian subgroup that is dense in $G$ with respect to the profinite topology. Then  $w\{ G \}$ is finite.
\end{theorem}

As a result, since linear groups are well-known to be equationally Noetherian (\textit{cf.} \cite[Theorem 3.5]{Bry77}), we extend Merzjalkov's Theorem as follows:

\begin{corollary}
    Every word is strongly concise in the class of profinite linear groups.
\end{corollary}

Furthermore, Theorem~\ref{thm: eq noetherian} also applies to pro-$\C$ completions of residually $\C$ linear groups, which might not be linear themselves. For instance, it is still unknown whether non-abelian free pro-$p$ groups are linear over any field (\textit{cf.} \cite[Conjecture 3.8]{LuSh94}).

\begin{corollary}
\label{cor: completion of linear}
    Every word is strongly concise in the class of groups consisting of pro-$\C$ completions of residually $\C$  linear groups. 
\end{corollary} 

Finitely generated abelian-by-polycyclic groups are residually finite (\textit{cf.}~\cite[\S 7.2]{LeRo04}) and equationally Noetherian (\textit{cf.} \cite[Theorem 1.3]{Val}). Thus Theorem \ref{thm: eq noetherian} implies that every word is strongly concise in any profinite group containing a dense subgroup of this kind. In the following theorem, we extend this result to (non-necessarily finitely generated) virtually abelian-by-polycyclic groups, although it remains unknown whether they are equationally Noetherian. 

\begin{theorem}
\label{thm: abelian-by-polycyclic}
Let $G$ be a profinite group, and suppose that $G$ has a dense virtually abelian-by-polycyclic subgroup.
Then every word is strongly concise in $G$.
\end{theorem}

As a consequence, we obtain:

\begin{corollary}
Every word is strongly concise in the class of:
\begin{itemize}
\item[\textup{(i)}] pro-$\C$ completions of virtually abelian-by-polycyclic groups,
\item[\textup{(ii)}] profinite virtually abelian-by-(finitely generated nilpotent) groups.
\end{itemize}
\end{corollary}

In the previous result,  \textit{finitely generated} is understood as topologically finitely generated. The following remains open:

\begin{question}
Is every word strongly concise in the class of virtually abelian-by-polyprocyclic groups?
\end{question}

The question falls outside the techniques of this paper, as a polyprocyclic group may not contain a dense polycyclic subgroup. For instance, for an odd prime $p$, the pro-$p$ group
$G = \Z_p \rtimes \Z_p$, topologically generated by $x$ and $y$, and with the action given by $x^y = x^2$, does not contain any dense metacyclic subgroup.  However, every word is strongly concise in the class of virtually polyprocyclic pro-$p$ groups, which includes the group in the previous example (\textit{cf.} \cite[Theorem 1.2]{JZ08}).

\medskip

Finally, abelian-by-nilpotent groups may not be nor equationally Noetherian (\textit{cf.}~\cite[Propositions 4.1 \&  5]{BMR}) nor abelian-by-polycyclic, and, therefore, the following remains open as well:

\begin{question}
Is every word strongly concise in the class of profinite virtually abelian-by-nilpotent groups?
\end{question}


\section{Preliminaries}
\label{sec: preliminaries}

Given a group $G$  and the free group on $k$ variables $ F_k = F(x_1, \dots, x_k)$, every element $w$ of the free-product $ G * F_k $ can be represented using the variables $ x_1, \dots, x_k $ and, if necessary, additional constants $ a_1, \dots, a_{\ell} $ from $ G $. That is,
\[
w = \hat{w}(a_1, \ldots, a_{\ell}, x_1, \ldots, x_k)
\]
for a suitable group word $ \hat{w}$  in $ l + k $ variables.
By analogy with word maps, $ w  $ defines in $G$ a so-called \emph{generalised word map}
\[
w \colon G^k \to G, \quad (g_1, \dots, g_k) \mapsto w(g_1, \dots, g_k) = \hat{w}(a_1, \dots, a_{\ell}, g_1, \dots, g_k).
\]
The \emph{verbal topology} on $ G^k $ is the topology whose sub-basis for the closed sets consists of the \emph{solution sets}
\[
V_G(S) = \left\{ (g_1, \dots, g_k) \in G^k \mid w(g_1, \dots, g_k) = 1 \text{ for all } w \in S \right\},
\]
where $ S \subseteq G * F_k $. We say that $ G $ is \emph{equationally Noetherian} if, for all $k$, any descending chain of verbally closed subsets in $G^k$ stabilizes (\textit{cf.}~\cite[Lemma 3.1]{Val} for alternative characterizations).
A straightforward consequence, which will play a central role in the article, is that the connected component of the identity in $G$ with respect to the verbal topology, also known as the \emph{irreducible component} of $ G$, is a normal irreducible subgroup of finite index (\textit{cf.} \cite[Lemma 5.2]{Weh}).

\medskip

Before delving further, we should establish some terminology. Given a set $X$ contained in two distinct groups $G^k$ and $H^k$, one may consider the verbal topology on both $G^k$ and $H^k$.
To avoid confusion, we will denote them respectively as the $G$-verbal and $H$-verbal topologies. Furthermore, in a profinite group $G$, these verbal topologies will be frequently used simultaneously with the original profinite topology of $G$.
Therefore, topological concepts will be clarified with a descriptor indicating the topology being used: profinite-continuous, verbally connected, $H$-verbally closed, etc.

Finally, given a subset $ X \subseteq G$, its \emph{profinite closure} in $G$ will be denoted by $ \pcl(X) $, while the \emph{verbal closure in $G$} will be denoted by $ \vvcl(X)$, or $\Gvcl(X)$ when we want to specify the group.

\medskip

We will require the following lemma regarding the verbal closure of finite Cartesian products: 

\begin{lemma}
\label{lem: product of closures}
Let $G$ be a group and let $L_1, \ldots, L_k \subseteq G$.
Then
\[ \vvcl(L_1 \times \cdots\times L_k) = \vvcl(L_1) \times\cdots\times \vvcl(L_k).\]
\end{lemma}
\begin{proof}
By an inductive argument, it suffices to prove the result for $k=2$.
    It is clear that $\vvcl(L_1\times L_2)\subseteq \vvcl(L_1)\times\vvcl(L_2)$.
    In order to prove the other inclusion, write 
    \[ \vvcl(L_1 \times L_2) = \bigcap_{i \in I} \bigg(\bigcup_{j \in J_i} V_G\left(w_{ij}(x_1, x_2) \right)\bigg)\]
for some set $I \subseteq \N$, finite sets $J_i \subseteq \N$ and generalised words $w_{ij}(x_1, x_2) \in G * \langle x_1, x_2 \rangle$.
For each $g \in L_2$, we have $w_{ij}(x_1, g) \in G * \langle x_1 \rangle$
and that
\[
\vvcl(L_1 ) \subseteq \bigcap_{i \in I}\bigg( \bigcup_{j \in J_i} V_G(w_{ij}(x_1, g))\bigg).
\]
Thus, $\vvcl(L_1) \times \{g\} \subseteq \vvcl(L_1 \times L_2),$ and since $g$ was arbitrary,
\[ \vvcl(L_1) \times L_2 \subseteq \vvcl(L_1 \times L_2).\]
Following the same line of reasoning with $g \in \vvcl(L_1)$, we prove that $\{ g \} \times \vvcl(L_2) \subseteq \vvcl(L_1 \times L_2)$, and we therefore obtain the desired equality.
\end{proof}

Let $w$ be a word in $k$ variables and $K,H \leq G$. We say that $K$ is \emph{marginal} for $w$ in $H$ if for every index $i \in \{1, \ldots, k\}$, every  $ g \in K$ and every $\h=(h_1,\ldots, h_k) \in H^k$ we have
\[ w(h_1,\ldots,h_{i-1},g h_i,h_{i+1},\ldots,h_k) = w(h_1,\ldots,h_{i-1},h_i g, h_{i+1}, \ldots, h_k) = w(\h),
\]
where $w(\h)$ stands for $w(h_1, \dots, h_k).$ 
In other words, $K$ is marginal for $w$ in $H$ if for every $i \in \{1, \dots, k\}$ and every tuple $\h = (h_1, \dots, h_k) \in H^{k}$, the maps
    \begin{align*}
        l_{\h}^i: G& \longrightarrow G\\
        g&\longmapsto w(h_1,\ldots, h_{i-1},g h_i, h_{i+1},\ldots, h_k) \cdot w(\h)^{-1}    
    \end{align*}
    and
    \begin{align*}
        r_{\h}^i:G&\longrightarrow G\\
        g&\longmapsto w(h_1,\ldots,h_{i-1}, h_ig, h_{i+1},\ldots, h_k) \cdot w(\h)^{-1}
    \end{align*}
  satisfy $l_\h^i(K) = r_\h^{i}(K) = \{ 1 \}$. 

\begin{lemma}
    \label{lem: closures preserve marginality}
    Let $G$ be a Hausdorff topological group and let $K,H$ be subgroups of $G$.
    If $K$ is marginal for a word $w$ in $H$, then the closure $\cl(K)$ of $K$ is marginal for $w$ in the closure $\cl(H)$ of $H$.   
\end{lemma}
\begin{proof}
Let $k$ be the number of variables of $w$, and fix $\h=(h_1, \ldots, h_k)$ in $H^k$. Since $K$ is marginal for $w$ in $H$, we have $l_{\h}^i(K) = \{1 \}$ for every $i\in\{1,\ldots,k\}$, and by the continuity of $l_{\h}^i$, since $G$ is Hausdorff, we obtain $l_{\h}^i(\cl(K)) \subseteq$ $\cl(l_\h^{i}(K)) =\{1\}$. In the same fashion $r_{\h}^{i}(\cl(K)) = \{ 1 \}$.   Since $\h$ was chosen arbitrarily in $H^k$, we conclude that $\cl(K)$ is marginal for $w$ in $H$.

Now fix $g \in \cl(K)$ and for each $i\in\{1,\ldots,k\}$ define the maps
\begin{align*}
\ell_{g}^i \colon & \quad G^k   & \longrightarrow \quad & G \\
                  & \quad (g_1, \ldots, g_k) &  \longmapsto \quad & w(g_1, \ldots, g_{i-1}, gg_i, g_{i+1}, \ldots, g_k) \, w(g_1, \ldots, g_k)^{-1}
\end{align*}
    and
\begin{align*}
\rcal_g^i \colon & \quad G^k   & \longrightarrow \quad & G \\
                  & \quad (g_1, \ldots, g_k) &  \longmapsto \quad & w(g_1, \ldots, g_{i-1}, g_i g, g_{i+1}, \ldots, g_k) \, w(g_1, \ldots, g_k)^{-1}.
\end{align*}
 Since $\cl(K)$ is marginal for $w$ in $H$, these maps satisfy $\ell_{ g}^i(H^k)=\rcal_{g}^i(H^k)=\{1\}$, and, as before, by continuity, we obtain $\ell^{i}_g\left(\cl(H)^k \right)=\rcal^{i}_ g(\cl(H)^k)= \{1\}$.
\end{proof}


\section{Proof of the main result}
\label{sec: eq noetherian}

In this section we prove Theorem \ref{thm: eq noetherian} and deduce Corollary \ref{cor: completion of linear}.
We start stating the following result, which is often key when working on strong conciseness.
Recall that a profinite space is a topological space that is Hausdorff, compact and totally disconnected.

\begin{proposition}[\mbox{\cite[Proposition 2.1]{DKS}}]
\label{prop: DKS}
Let $F \colon X \rightarrow Y$ be a continuous map between non-empty profinite spaces, and suppose that $|F(X)| < 2^{\aleph_0}$.
Then, there exists a non-empty profinite-open subset $U \subseteq X$ such that $F|_U$ is constant.
\end{proposition}

\begin{proposition}
\label{prop: connected implies marginal}
Let $G$ be a profinite group and $w$ a word such that $|w\{G\}| < 2^{\aleph_0}$.
Let $N\trianglelefteq H\le G$. If $N$ is $H$-verbally irreducible, then $\pcl(N)$ is marginal for $w$ in $\pcl(H)$.
\end{proposition} 

\begin{proof}
Let $k\in\N$ be the number of variables of the word $w$.
We fix a tuple $\x $ in $ H^k$, and we define $w_{\x}$ to be the generalised word map
\begin{align*}
    w_{\x} \colon G^k &\longrightarrow G\\
    \g\ &\longmapsto w(\x\cdot \g) w(\x)^{-1},
\end{align*}
where $\x \cdot \g$ denotes the component-wise product of tuples. Observe that $w_{\x}$ is profinite-continuous and that $|w_{\x}\{G\}|<2^{\aleph_0}$.

Let $\mathcal{N}=\pcl(N)$ be the profinite closure of $N$ in $G$. Since $N$ is normal in $H$, the subgroup $\mathcal{N}$ is also normalised by $H$, and thus $w_{\x}\{ \mathcal{N} \}\subseteq \mathcal{N}$. Hence, in view of Proposition~\ref{prop: DKS}, there exists a tuple $\n =(n_1,\ldots, n_k)$ in $N^k$, and a profinite-open subgroup $\mathcal{M} \leq \mathcal{N}$ (in particular $|\mathcal{N}: \mathcal{M}|$ is finite) such that $w_{\x}$ is constant in $n_1 \mathcal{M}\times \cdots \times n_k\mathcal{M}$, that is,
$$
w_{\x}(\n \cdot \g ) = w_{\x}(\n)\in N
$$
for every $\g \in \mathcal{M}^k$. 

Since $\x\in H^k$, the restriction $w_{\x}|_{H^k} \colon H^k \rightarrow H $
is $H$-verbally continuous, and so, being $H$ a $T_1$ topological space with respect to the $H$-verbal topology, $w_{\x}$ is constant in 
\begin{equation}
    \label{eq: closure of products}
    \vcl\left(\bigtimes_{i=1}^k n_i (\mathcal{M} \cap N) \right) = \bigtimes_{i=1}^k \vcl(n_i (\mathcal{M} \cap N)),
\end{equation}
where the equality follows from Lemma~\ref{lem: product of closures}.

We claim that $\vcl(t(\mathcal{M} \cap N)) = \vcl(N)$ for every $n \in N$. Fix one $n\in N$, and let $T$ be a \textit{finite} left transversal for $\mathcal{M}\cap N$ in $N$ such that $n\in T$.
As $N\le H$, the left multiplication map $L_t\colon H \rightarrow H$, $x \mapsto tx$ is an $H$-verbal homeomorphism for every $t\in T$, so that 
\[ \vcl(N)
=
\vcl\bigg(\bigcup_{t \in T} t (\mathcal{M} \cap N)\bigg)
=
\bigcup_{t \in T} t \vcl(\mathcal{M} \cap N).\]
Now, since $N$ is $H$-verbally irreducible, so is $\vcl(N)$, and since the $t\vcl(\mathcal{M}\cap N)$ are $H$-verbally closed, it follows that
$$
\vcl(t(\mathcal{M} \cap N))
= 
t\vcl(\mathcal{M} \cap N)
=
\vcl(N),
$$
as claimed. 

This shows, in view of (\ref{eq: closure of products}), that $w_{\x}$ is constant in $\vcl(N)^{k}$, i.e.
\[ w_\x\left\{ \vcl(N) \right\} = w_{\x}(1, \dots, 1) = \{ 1 \}.\]
Since $\x$ was arbitrary, $\vcl(N)$ is marginal for $w$ in $H$,  and we conclude by applying  Lemma \ref{lem: closures preserve marginality}.
\end{proof}

\begin{proof}[Proof of Theorem~\ref{thm: eq noetherian}]
Let $H$ be a profinite-dense equationally Noetherian subgroup of $G$, and let $H_0$ be the irreducible component of $H$, which turns out to be a normal subgroup of finite index in $H$ (\textit{cf.}~\cite[Lemma 5.2]{Weh}).
On the one hand, since $H$ is profinite-dense in $G$, the subgroup $\pcl(H_0)$ has finite index in $\pcl(H)=G$.
On the other hand, Proposition~\ref{prop: connected implies marginal} shows that $\pcl(H_0)$ is marginal for $w$ in $ G$.
As $|w\{ G \}| \leq |G : \pcl(H_0)|^k$, where $k$ is the number of variables in $w$, it follows that $w\{G\}$ is finite. 
\end{proof}

By combining the theorem above and the following simple lemma, Corollary~\ref{cor: completion of linear} follows immediately, as a residually $\C$ group has a dense embedding into its pro-$\C$ completion.

\begin{lemma}
\label{lem: sequences}
Let $G$ be a Hausdorff topological group and let $H \leq G$ be a sequentially dense subgroup of $G$.
If $w\{H\}$ is finite for a word $w$, then $w(G) = w(H)$.
\end{lemma}
\begin{proof}
For each tuple $(g_1, \dots, g_k)$ in $G^k$ there exist sequences $(h_{i, \, n})_{n \in \mathbb{N}}$ in $H$, with $i \in \{1, \dots, k\}$, 
such that $(h_{i,n})_{n \in \mathbb{N}}$ 
converges to $g_i$.
Then, by the continuity of word maps, it follows that
\[ \left(w\left(h_{1,n}, \dots, h_{k,n} \right) \right)_{n \in \mathbb{N}} \rightarrow w(g_1, \dots, g_k)\]
is convergent.
Since $w\{H \}$ is finite and $G$ is a Hausdorff topological space, the sequence $(w(h_{1,n}, \dots, h_{k,n}))_{n \in \mathbb{N}}$ eventually stabilises, i.e.
\[ w\left(h_{1,n}, \dots, h_{k,n} \right) = w(g_1, \dots, g_k) \]
for every large enough $n$. Consequently, $w\{G\} = w\{H \}$, and so $w(G) = w(H)$.
\end{proof}


\section{Virtually abelian-by-polycyclic groups}
\label{sec: abelian-by-polycyclic}

In this last section, we prove Theorem \ref{thm: abelian-by-polycyclic}. The following lemma will be useful.

\begin{lemma}
\label{lem: existence of abelian normal subgroup}
    Let $G$ be a virtually abelian-by-polycyclic group. Then there exists an abelian normal subgroup $A$ of $G$ such that $G/A$ is virtually polycyclic.
\end{lemma}
\begin{proof}
    Let $H$ be a finite-index normal abelian-by-polycyclic subgroup of $G$, and let $B \trianglelefteq H$ be an abelian subgroup such that $H/B$ is polycyclic.
    Write $n=|G:H|$ and let $g_1,\ldots,g_n\in G$ be representative elements for each right coset of $H$ in $G$. Being $B$ normal in $H$, the subgroup $A=\cap_{i=1}^n B^{g_i}$ is abelian and normal in $G$.
    We will show that $G/A$ is virtually polycyclic.

    For that purpose define recursively $A_0 = B$ and $A_i=A_{i-1}\cap B^{g_i}$, for $1\le i\le n$.  Note that each $A_i$ is a normal subgroup of $H$ and that $A_n=A$. It thus suffices to show that each $H/A_i$ is polycyclic.
    We already know it for $A_0 = B$, so assume by induction that  $H/A_{i-1}$ is polycyclic for $i \geq 1$. Thus,  the subgroup $B^{g_i}A_{i-1}/A_{i-1}$ of $H/A_{i-1}$ is polycyclic too, and hence $B^{g_i}/A_i$ is polycyclic.
    Since the homomorphic image $H/B^{g_i}$ of $H/B$ is also polycyclic, then $H/A_i$ is itself polycyclic. 
\end{proof}

The proof of the following proposition follows \cite{Bry77}.

\begin{proposition}
\label{prop: connected inside abelian}
    Let $G$ be a group and $N \trianglelefteq G$ an abelian normal subgroup of $G$.
    Suppose that the group ring $\mathbb{Z}[G/C_G(N)]$ is right Noetherian. Then $N$ satisfies the descending chain condition on $G$-verbally closed subsets.  
\end{proposition}
\begin{proof}
In view of \cite[Lemma 3.1]{Bry77}, we shall show that $N$ satisfies the descending chain condition on $G$-verbal solution sets. We can regard $N$ as a right $\mathbb{Z}[G/C_G(N)]$-module, so that the equation $v(h)=1$, where $v\in G*\langle x\rangle$ and $h\in N$, is equivalent to the condition $h^{\delta_v}=0$ for an element $\delta_v$ in $\mathbb{Z}[G/C_G(N)]$. Therefore, for every $S \subseteq G * \langle x \rangle$,
\begin{align*}
V_N(S) &= \left\{ h \in N \mid v(h) = 1 \ \text{for every}\ v \in S \right\}\\
&= \big\{ h \in N \mid h^{\delta_v} = 0\ \text{for every}\ v \in S \big\}\\
&= \big\{ h \in N \mid h^\delta = 0 \ \text{for every}\ \delta \in \Delta_S \big\},
\end{align*}
where $\Delta_S$ is the right ideal of $\mathbb{Z}[G/C_G(N)]$ generated by $\{\delta_v \mid v \in S \}$. 
Thus, every descending chain of solution sets
\begin{equation}
\label{eq: chain solution sets}
V_N(S_1) \supseteq \dots \supseteq V_N(S_i) \supseteq \cdots,
\end{equation}
where $S_i \subseteq G * \langle  x \rangle,$ gives rise to an ascending chain 
\begin{equation}
\label{eq: chain ideals}
\Delta_{S_1} \subseteq \cdots \subseteq \Delta_{S_i} \subseteq \cdots
\end{equation}
of right ideals of $\mathbb{Z}[G/C_G(N)]$. 
Since $\mathbb{Z}[G/C_G(N)]$ is right Noetherian, both chains \eqref{eq: chain solution sets} and \eqref{eq: chain ideals} are stationary.
\end{proof}

\begin{corollary}
\label{cor: finite index connected}
Let $G$ be a virtually abelian-by-polycyclic group, and let $A \trianglelefteq G$ be an abelian normal subgroup of $G$ such that $G/A$ is virtually polycyclic (\textit{cf.} Lemma \ref{lem: existence of abelian normal subgroup}).
Let $A_0 \leq A$ be the $G$-verbally connected component of the identity in $A$. Then $A_0$  has finite index in $A$ and is normal in $G$.
\end{corollary}

\begin{proof}
Since $\mathbb{Z}[G/A]$ is right Noetherian (\textit{cf.} \cite[\S \,4.2.3]{LeRo04}), Theorem \ref{prop: connected inside abelian} shows that $A$ satisfies the descending chain condition on $G$-verbally closed subgroups, and thus $|A : A_0|$ is finite by \cite[Lemma 5.2]{Weh}.
Moreover, $A_0$ is invariant under $G$-verbally continuous automorphisms, so in particular it is invariant under the conjugation by any element $g \in G$.
\end{proof}

\begin{proof}[Proof of Theorem \ref{thm: abelian-by-polycyclic}]

Let $H$ be a dense virtually abelian-by-polycyclic subgroup of $G$. Suppose that $|w\{G\}| < 2^{\aleph_0}$ for some word $w$.
Since every word is concise in the class of virtually abelian-by-polycyclic groups, by Lemma~\ref{lem: sequences}, it is enough to prove that $w\{G\}$ is finite.

Let $A$ be an abelian normal subgroup of $H$ such that $H/A$ is virtually polycyclic (\textit{cf.} Lemma~\ref{lem: existence of abelian normal subgroup}).
By Corollary \ref{cor: finite index connected}, we may assume $A$ is $H$-verbally connected, so applying Proposition \ref{prop: connected implies marginal}, we obtain that $\pcl(A)$ is marginal for $w$ in $G$.

Let $T\subseteq H$ be a finite subset of $H$ such that $H=\langle T\rangle A$, and write $\mathcal{T}=\pcl(\langle T\rangle)$. Note that $\langle T\rangle$ is a finitely generated virtually abelian-by-polycyclic group, so it is equationally Noetherian by \cite[Theorem 1.3]{Val} and \cite[Theorem 1]{BM}.
Hence, $|w\{\mathcal{ T}\}|<\infty$ by Theorem \ref{thm: eq noetherian}. On the other hand, $G=\mathcal{T}\cdot \pcl(A)$, and since $\pcl(A)$ is marginal for $w$ in $G$, we conclude that $w\{G \} = w\{\mathcal{T}\}$.
\end{proof}

\end{document}